\documentclass[12pt]{article}

        \usepackage[reqno, namelimits, sumlimits]{amsmath}
         \usepackage{amssymb, amsfonts}
         \usepackage{amsthm}
\usepackage{color}

 \newtheorem{theorem}{Theorem}[section]

 \newtheorem{remarks}[theorem]{Remarks}
 
 \newtheorem{pro}[theorem]{Proposition}

\newcommand{\beq}{\begin{equation}}
\newcommand{\eeq}{\end{equation}}
\newcommand{\ben}{\begin{eqnarray}}
\newcommand{\een}{\end{eqnarray}}
\newcommand{\beno}{\begin{eqnarray*}}
\newcommand{\eeno}{\end{eqnarray*}}

\title{ Sufficient conditions on Liouville type theorems for the 3D steady Navier-Stokes equations
}
\author{G. Seregin\footnote{Oxford University, UK}  \footnote{PDMI, RAS, Russia}, W. Wang\footnote{Dalian University of Technology, China}}

\begin{document}

\maketitle

\begin{abstract}
Our aim is to prove Liouville type theorems for the three dimensional steady-state Navier-Stokes equations provided the velocity field belongs to some Lorentz spaces. The corresponding statement contains  several known results as a particular case.
\end{abstract}


\setcounter{equation}{0}
\section{Introduction}

The classical Liouville type problem is to describe all bounded solutions to the three dimensional steady-state Navier-Stokes equations
\begin{equation} \label{eq:SNS}
-\Delta u+u\cdot \nabla u=-\nabla{p},\qquad{\rm div}\, u=0
\end{equation}
in the entire space $\mathbb R^3$. This is still an open problem.

Another Liouville type problem is to show that all solutions to system (\ref{eq:SNS})
belonging to the space ${\stackrel{\circ}J}{^1_2}$, which is the closure of the set of all smooth divergence free compactly supported functions, denoted by $C^\infty_{0,0}(\mathbb R^3)$,
with respect to the semi-norm
$$\|\nabla u\|_{L_2(\mathbb R^3)}=\Big(\int\limits_{\mathbb R^3}|\nabla u|^2dx\Big)^\frac 12,$$
are identically equal to zero. This problem is related to the name of J. Leray and, to the best of  authors's knowledge, has not been solved yet.

However, there are several sufficient conditions providing that all solutions $u$ to (\ref{eq:SNS}) are equal zero. Let us list the most interesting ones.

We start with Galdi's result. Galdi proved the above Liouville type theorem under the assumption that
$$u\in L^{\frac92}(R^3)$$ in  \cite{Galdi}. Another interesting result belongs to Chae.
 In \cite{Chae}, he showed the condition $$\triangle u\in L^{\frac65}(R^3)$$ is sufficient for  $u\equiv0$ in $\mathbb R^3$. Also, Chae-Wolf gave a logarithmic   improvement of  Galdi's result in \cite{ChaeWolf},
 assuming that
 $$N(u):=\int\limits_{\mathbb R^3} |u|^{\frac92}\{\ln(2+{1}/{|u|})\}^{-1}dx<\infty.$$

 Let us notice two other sufficient conditions. It has been  shown in \cite{Se} that the condition $$u\in BMO^{-1}(\mathbb R^3)$$  implies $u\equiv 0$ as well. Moreover, Kozono, et al., proved in \cite{KTW} that $u\equiv0$ if the vorticity $$w=o(|x|^{-\frac53})$$
 for sufficiently large $|x|$ or
$$\|u\|_{L^{\frac92,\infty}(\mathbb R^3)}\leq \delta D(u)^{1/3}$$ for a small constant $\delta$. More references, we refer to \cite{ChaeWeng,KNSS,Se2} and the references therein.

One of our aims is to relax the restriction imposed on
  the norm $\|u\|_{L^{\frac92,\infty}(\mathbb R^3)}$ 
  in \cite{KTW}. Let us remind the definition of the Lorentz spaces. 

Suppose that $\Omega\subseteq \mathbb R^n$ and $1\leq p<\infty$, $1\leq \ell \leq \infty$.
It is said that a measurable function $f$ belongs to the Lorentz space $L^{p,\ell}(\Omega)$ if $\|f\|_{L^{p,\ell}(\Omega)}<+\infty$, where
$$
\|f\|_{L^{p,\ell}(\Omega)}:=
\left\{\begin{array}{lll}
\Big(\int\limits_0^{\infty}\sigma^{\ell-1}|\{x\in \Omega:|f(x)|>\sigma\}|^{\frac{\ell}{p}}d\sigma\Big)^{\frac 1{\ell}}\quad
\textrm{if  } \ell<+\infty,\\
\displaystyle\sup_{\sigma>0}\sigma|\{x\in \Omega:|f(x)|>\sigma\}|^{\frac{1}{p}}\quad
\textrm{if  } \ell=+\infty.
\end{array}\right.
$$


Given $u$, define the following quantity
$$
M_{\gamma,q,\ell}(R):= R^{\gamma-\frac3q}\|u\|_{L^{q,\ell}(B_R\setminus B_{R/2})}
$$
where $B(R)=B(0,R)$.

Our result is as follows:
\begin{theorem}\label{thm:NS Morrey}
Let $u$ and $p$ be a smooth solution  to (\ref{eq:SNS}).

(i)~For $q>3$, $3\leq\ell\leq \infty$(or $q=\ell=3$), $\gamma =\frac23$, assume that
\begin{equation}\label{boundedness}
	\liminf_{R\rightarrow\infty}M_{\frac23,q,\ell}(R)<\infty\end{equation}
then
\begin{equation}\label{energy bound}
D(u):=\int\limits_{\mathbb R^3}|\nabla u|^2dx\leq c(q,\ell)
\liminf_{R\rightarrow\infty}M^3_{\frac23,q,\ell}(R).\end{equation}
Moreover, if 
\begin{equation}\label{CaseI}
	\liminf_{R\rightarrow\infty}M^3_{\frac23,q,\ell}(R)\leq \delta D(u)
\end{equation}
for some $0<\delta<1/c(q,\ell)$, then $u\equiv 0$.

(ii)~For $12/5< q<3$,  $1\leq\ell\leq \infty$, $\gamma>\frac13+\frac{1}{q}$, suppose that
\begin{equation}\label{CaseII}
\liminf\limits_{R\rightarrow\infty}M_{\gamma,q,\ell}(R)<\infty\end{equation} holds then $u\equiv 0$ as well.
\end{theorem}

\begin{remarks}

\noindent
(i)~Letting $q=\ell=\frac 92$ and assuming that $u\in L_\frac 92(\mathbb R^3)$, we observe that $M_{\frac 23, \frac 92,\frac 92}(R)\to 0$ as $R\to\infty$. So, Galdi's result follows from Theorem \ref{thm:NS Morrey}.

\noindent
(ii)  For  $q=\frac92$ and $\ell=\infty$, condition (\ref{CaseI}) can be regarded as  a generalisation  a result proved by Kozono-Terasawa-Wakasugi in \cite{KTW}.

\noindent
(iii) If we let 
$N(u)=\infty$,
then $N(v)=\infty$ for the function $ v=|x|^{-\frac23}[\ln\ln(|x|+e)]^{-\nu}
$ with  $0<\nu\leq \frac29$. However, if we assume that $$|u|\leq \frac{C}{|x|^{\frac23}[\ln\ln(|x|+e)]^{\nu}}$$ for the same $\nu$, we can easily check the following fact
$$
\|u\|_{L^{\frac92,\infty}(B(R)\setminus B(R/2))}\rightarrow0$$ as $R\rightarrow\infty$.
The latter shows  that statement (i) of Theorem \ref{thm:NS Morrey} does not follow from 
the result of Chae-Wolf \cite{ChaeWolf}.

\noindent
(iii) The second statement of the theorem is an improvement of one of the results in \cite{Se3}, see Theorem 1.8, where it is assume that $\gamma >\frac {4q-3}{6q-3}$.

\end{remarks}

\setcounter{equation}{0}\section{Caccioppoli Type Inequalities }
We start with an auxiliary lemma about Caccioppoli type inequality for the system
(\ref{eq:SNS}).

\begin{pro}\label{lem:estimate of LQ}
Let $u$ and $p$ be the smooth solution of (\ref{eq:SNS}). Then the following Caccioppoli type inequalities hold: %

if $q>3$ and $1\leq \ell\leq \infty$, then
\begin{equation}
\label{eq:increasing estimate-Lorentz}
\int\limits_{B(R/2)}|\nabla u|^2dx \leq CR^{-2}\int\limits_{B(R)\setminus B(R/2)}|{u}|^2dx+CR^{2-\frac9q}\|{u}\|^3_{L^{q,\ell}(B(R)\setminus B(R/2))};
\end{equation}


\noindent
 if $0<\delta\leq 1$, $3>q>\frac{ 6(3-\delta)}{6-\delta}$, 
then
$$\int\limits_{B(R/2)}|\nabla u|^2dx\leq \frac{C}{R^2}\int\limits_{B(R)\setminus B(R/2)}|{u}|^2dx+$$
\begin{equation}
\label{eq:increasing estimate-Lorentz2}
+C(\delta)\Big(\|{u}\|_{L^{q,\infty}2(B(R)\setminus B(R/2))}^{3-\delta}R^{2-\frac{9-3\delta}{q}-\frac \delta2}\Big)^\frac 2{2-\delta}.
\end{equation}
\end{pro}
\begin{proof} Given $R>0$, fix numbers $\varrho$ and $r$ so that $3R/4\leq\varrho<r\leq R$.
	Now, let us pick up a cut-off function  $\phi(x)\in C_0^\infty(B(R))$ satisfying  the following conditions: $0\leq \phi\leq 1$, $\phi(x)=1$ if $x\in B(\varrho)$, $\phi(x)=0$ if $x\in B(r)^c$, and $|\nabla \phi(x)|\leq c/(r-\varrho)$.
	
	We also may assume that function $\phi(x)=\eta(|x|)$, i.e., it depends  on the distance to the origin only. In this case, it is easy to check that
	$$\int\limits_{B(r)\setminus B(2r/3)}\nabla \phi\cdot udx=0.
	$$
	Then, by Theorem III 3.4 in \cite{Galdi} and by scaling, for any $1< s<\infty$, there exist a constant $c_0(s)$ and a function $
w\in W^1_s(B(r))$ such that
$div~w=\nabla\phi\cdot {u}$ in $B(r)$, $w=0$ on $\partial B(r)\cup \partial B(2r/3)$, and
$$
\int\limits_{B(r)\setminus B(2r/3)}|\nabla w|^sdx\leq C_0(s)\int\limits_{B(r)\setminus B(2r/3)}|\nabla\phi\cdot { u}|^sdx.
$$
The function $w$ is extended by zero outside the set $B(r)\setminus B(2r/3)$.
Moreover, it is actually smooth as $u$ is smooth.

According to the general Marcinkiewicz interpolation theorem, we find

$$
\|\nabla w\|_{L^{q,\ell}(B(r)\setminus B(2r/3))}
\leq C_0(q) \|\nabla\phi\cdot { u}\|_{L^{q,\ell}(B(r)\setminus B(2r/3))}.
$$
for any $1<q<\infty$ and any $1\leq \ell\leq \infty.$

Multiplication of  both sides of the equation (\ref{eq:SNS}) by $(\phi {u}-w)$ and integration  by parts give:

$$
\int\limits_{B(r)}\phi|\nabla u|^2dx=$$$$ =-\int\limits_{B(r)}\nabla u : (\nabla \phi\otimes{u})  dx+\int\limits_{B(r)}\nabla w:\nabla u  dx
-\int\limits_{B(r)}\nabla u : (\phi {u}\otimes u) dx+$$$$+\int\limits_{B(r)}\nabla u :(u\otimes w) dx=
 I_1+\cdots+I_4.
$$

Obviously, since $R\geq r>\varrho\geq 3R/4>R/2$,
$$
|I_1|\leq C\frac{1}{r-\rho}\Big(\int\limits_{B(r)}|\nabla u|^2dx\Big)^{\frac12}\Big(\int\limits_{B(r)\setminus B(\varrho)}| {u}|^2dx\Big)^{\frac12}\leq$$
$$C\frac{1}{r-\rho}\Big(\int\limits_{B(r)}|\nabla u|^2dx\Big)^{\frac12}\Big(\int\limits_{B(R)\setminus B(R/2)}| {u}|^2dx\Big)^{\frac12}$$
and
$$
|I_2| \leq C\Big(\int\limits_{B(r)}|\nabla u|^2dx\Big)^{\frac12}\|\nabla w\|_{L_{2}(B(r)\setminus B(2r/3))}\leq $$$$\leq C \frac{1}{r-\rho}\|\nabla u\|_{L_2(B(r))} \|{u}\|_{L_2(B(R)\setminus B(R/2))}.
$$

 Now, our aim is to prove inequality (\ref{eq:increasing estimate-Lorentz}). To this end, assuming that $q>3$ and $\ell\geq 3$, let us estimate $I_3$, using integration by parts and H\"older inequality in Lorentz spaces. Indeed,
$$
|I_3|= \frac12\Big|\int\limits_{B(r)\setminus B(\varrho)}{u}\cdot\nabla\phi |{u}|^2dx\Big|\leq$$$$\leq  C\|u\cdot\nabla \phi \|_{L^{\frac q{q-2},\frac \ell {\ell-2}}(B(r)\setminus B(\varrho))}
\||u|^2\|_{L^{\frac q2,\frac \ell2}(B(r)\setminus B(\varrho))}\leq $$$$\leq C\frac 1{r-\varrho}\|u\|^3_{L^{q,\ell}(B(R)\setminus B(R/2))}\|1\|_{L^{\frac q{q-3},\frac \ell{\ell-3}}(B(R)\setminus B(R/2))}\leq $$$$\leq C\frac{1}{r-\rho}\|{u}\|_{L^{q,\ell}(B(R)\setminus B(R/2))}^3R^{3-\frac9q}.
$$
The quantity $I_4$ is evaluated similarly, if we use the estimate for the gradient of $w$ with suitable exponents:
$$
|I_4|= |\int\limits_{B(r)\setminus B(2r/3)}\nabla w : u\otimes u dx|\leq
$$$$
\leq  C\|\nabla w \|_{L^{\frac q{q-2},\frac \ell {\ell-2}}(B(r)\setminus B(2r/3))}
\||u|^2\|_{L^{ \frac q2, \frac \ell2}(B(r)\setminus B(2r/3)}\leq
$$
$$
\leq  C\|u\cdot\nabla \phi \|_{L^{\frac q{q-2},\frac \ell {\ell-2}}(B(r)\setminus B(2r/3))}
\|u\|^2_{L^{ q, \ell}(B(r)\setminus B(2r/3))}\leq
$$
$$\leq \frac{C}{\tau-\rho}\|{u}\|^3_{L^{q,\ell}(B(R)\setminus B(R/2))}R^{3-\frac9q}
.$$

Hence, we get
$$
\int\limits_{B(\rho)}|\nabla u|^2dx\leq\frac12\int\limits_{B(r)}|\nabla u|^2dx+\frac{C}{(r-\rho)^2}\Big(\int\limits_{B(R)\setminus B(R/2)}|{u}|^2dx\Big)+$$
$$+\frac{C}{r-\rho}\|{u}\|^3_{L^{q,\ell}(B(R)\setminus B(R/2))}R^{3-\frac9q}
,$$
which yields the inequality (\ref{eq:increasing estimate-Lorentz}) by the standard iteration.



Now, let us prove the second inequality of the proposition. To this end, we introduce $\bar u=u-[u]_{B(r)\setminus B(2r/3)}$, where $[u]_{\Omega}$ is the mean value of $u$ over a  domain $\Omega$. Applying integration by parts, we find
$$I_3=-\frac 12\int\limits_{B(r)}\phi u\cdot\nabla (|u|^2)dx=-\frac 12\int\limits_{B(r)}\phi u\cdot\nabla (|u|^2-|[u]_{B(r)\setminus B(2r/3)}|^2)dx=
$$
$$=\frac 12\int\limits_{B(r)\setminus B(\varrho)}(u\cdot\nabla \phi)(|u|^2-|[u]_{B(r)\setminus B(2r/3)}|^2)dx$$
and, since $2r/3<3R/4\leq \varrho$,
$$|I_3|\leq \frac C{r-\varrho}\int\limits_{B(r)\setminus
B(2r/3)}|u||\bar u||u+[u]_{B(r)\setminus B(2r/3)}|dx.$$
Under our assumptions on numbers $q$ and $\delta$,
 the following is true
$$0<\beta =1-\frac {3-\delta}q -\frac \delta 6<1.
$$
So, applying the H\"older inequality for Lorentz spaces, we show
$$|I_3|\leq  \frac C{r-\varrho}\int\limits_{B(r)\setminus
B(2r/3)}|u||\bar u|^{1-\delta}|\bar u|^\delta|u+[u]_{B(r)\setminus B(2r/3)}|dx\leq $$
$$\leq
\frac C{r-\varrho}\|u\|_{L^{q,\infty}(B(r)\setminus B(2r/3))}\||\bar u|^{(1-\delta)}\|_{L^{\frac q{1-\delta},\infty}(B(r)\setminus B(2r/3))}\||\bar u|^\delta\|_{L_\frac 6\delta(B(r)\setminus B(2r/3))}\times
$$
$$\times\|1\|_{L^{\frac{1}{\beta},\frac 6{6-\delta}}(B(r)\setminus B(2r/3))}\|u+[u]_{B(r)\setminus B(2r/3)}\|_{L^{q,\infty}(B(r)\setminus B(2r/3))}\leq $$
$$\leq
\frac C{r-\varrho}
\|u\|_{L^{q,\infty}(B(r)\setminus B(R/2))}\|\bar u\|^{1-\delta}_{L^{ q,\infty}(B(r)\setminus B(2r/3))}\|\bar u\|^\delta_{L_6(B(r)\setminus B(2r/3))}\times
$$
$$\times R^{3\beta}
\|u+[u]_{B(r)\setminus B(2r/3)}\|_{L^{q,\infty}(B(r)\setminus B(2r/3))}.$$
By Gagliardo-Nireberg-Sobolev inequality and by the inequality
$$\|[u]_{B(r)\setminus B(2r/3)}\|_{L^{q,\infty}(B(r)\setminus B(2r/3))}\leq c\|u\|_{L^{q,\infty}(B(r)\setminus B(2r/3))},$$
we can transform the estimate of $|I_3|$ to the following final form
$$|I_3|\leq \frac C{r-\varrho}R^{3\beta}\|u\|^{3-\delta}_{L^{q,\infty}(B(r)\setminus B(R/2))}\|\nabla u\|^\delta_{L_2(B(r)\setminus B(2r/3))
}\leq
$$
$$\leq \frac C{r-\varrho}R^{3\beta}\|u\|^{3-\delta}_{L^{q,\infty}(B(r)\setminus B(R/2))}\|\nabla u\|^\delta_{L_2(B(r)\setminus B(R/2))
}\leq
$$
$$\leq \frac19\int\limits_{B(r)\setminus B(R/2)}|\nabla u|^2dx+C(\delta)\Big(\frac {R^{3\beta}}{r-\varrho}\|u\|^{3-\delta}_{L^{q,\infty}(B(r)\setminus B(R/2))}\Big)^\frac 2{2-\delta}.$$

Now, our aim is to evaluate $I_4$. Using similar arguments,  we have
$$
|I_4|= \Big|\int\limits_{B(r)\setminus B(2r/3)}(u\cdot\nabla u) \cdot w dx\Big|=\Big|\int\limits_{B(r)\setminus B(2r/3)}(u\cdot\nabla w) \cdot \bar{u} dx\Big|\leq $$
$$\leq\|\nabla w\|_{L^{q,\infty}(B(r)\setminus B(2r/3))}\|\bar u\|^{1-\delta}_{L^{ q,\infty}(B(r)\setminus B(2r/3))}\|\bar u\|^\delta_{L_6(B(r)\setminus B(2r/3))}\times
$$
$$\times R^{3\beta}
\|u
\|_{L^{q,\infty}(B(r)\setminus B(2r/3))}.$$
Taking into account the bound for the gradient of $w$, we arrive at the same type estimate as in the case of $I_3$

Consequently, combining bounds of $I_1,\cdots,I_4$ we get
$$
\int\limits_{B(\varrho)}|\nabla u|^2dx\leq\frac12\int\limits_{B(r)}|\nabla u|^2dx+\frac{C}{(r-\rho)^2}\Big(\int\limits_{B(R)\setminus B(R/2)}|{u}|^2dx\Big)$$$$+C(\delta)\Big(\frac{R^{3\beta
}}{(r-\varrho)}\|{u}\|^{3-\delta}_{L^{q,\infty}(B(R)\setminus B(R/2))}\Big)^{\frac{2}{2-\delta}}
$$
for any $\frac{3}{4}R\leq \varrho<\tau\leq R$.
Hence, the inequality (\ref{eq:increasing estimate-Lorentz2}) follows. 
\end{proof}

\setcounter{equation}{0}
\section{Proof of Theorem \ref{thm:NS Morrey}}

We start with a proof of the statement (i).
It is easy to check that, for
$2<q<6$, the following estimate is valid:
$$
R^{-2}\Big(\int\limits_{B(R)\setminus B(R/2)}|{u}|^2dx\Big)\leq C(q) R^{1-\frac6q}\|{u}\|_{L^{q,\infty}(B(R)\setminus B(R/2))}^2
\leq
$$
$$
\leq C(q)R^{-\frac 13}M^2_{\frac 23,q,\ell}(R)
.$$
Taking into account condition (\ref{boundedness}), we find (\ref{energy bound}) and then (\ref{CaseI}).

Now, our goal is to prove the statement (ii). Applying the H\"older inequality to the first term on the right hand side in (\ref{eq:increasing estimate-Lorentz2}), we find the following:
$$\int\limits_{B(R/2)}|\nabla u|^2dx\leq {C}{R^{\frac 13-\frac 2q}}M^2_{\frac 13+\frac 1q,q,\infty}(R)+C(\delta)\Big(\|{u}\|_{L^{q,\infty}2(B(R)\setminus B(R/2))}^{3-\delta}R^{3\beta -1}\Big)^\frac 2{2-\delta}\leq $$
\begin{equation}
\label{eq:increasing estimate-Lorentz2}
\leq {C}{R^{\frac 13-\frac 2q}}M^2_{\frac 13+\frac 1q,q,\infty}(R)+C\Big(M^{3-\delta}_{\gamma,q,\infty}(R)R^{3\beta -1-(\gamma-\frac 3q)(3-\delta)}\Big)^\frac 2{2-\delta}=
\end{equation}
$$= {C}{R^{\frac 13-\frac 2q}}M^2_{\frac 13+\frac 1q,q,\infty}(R)+C\Big(M^{3-\delta}_{\gamma,q,\infty}(R)R^{2-\frac \delta 2 -\gamma(3-\delta)}\Big)^\frac 2{2-\delta}.$$

Now, fix $q\in ]12/5,3[$. Then we can find $q_1$ having the following properties:
$$q>q_1>\frac {12}5,\qquad \gamma>\frac 13+\frac 1{q_1}>\frac 13+\frac 1q.$$
Given $q_1$, there exists a number $\delta\in ]0,1[$ such that
$$q_1=\frac {6(3-\delta)}{6-\delta}<q.$$
It remains to notice that
$$a:=2-\frac \delta 2-\gamma (3-\delta)=2-\frac {3(3-q_1)}{6-q_1}-\gamma\Big(3-\frac {6(3-q_1)}{6-q_1}\Big)=$$
$$=\frac {3+q_1-3q_1\gamma}{6-q_1}.$$
But $\gamma >\frac 13+\frac 1q_1$ and thus
$$a=3q_1\frac {\frac 13+\frac 1q_1-\gamma}{6-q_1}<0
$$
Passing to the limit as $R\to \infty$, we complete the proof of the theorem.

\noindent {\bf Acknowledgments.}
The first author is supported by the grant RFBR 17-01-
00099-a.
W. Wang was supported by NSFC under grant 11671067,
 "the Fundamental Research Funds for the Central Universities" and China Scholarship Council.

\end{document}